\documentclass[final,reqno]{amsart}

\usepackage{graphicx}
\usepackage{amsmath,amsthm,amsfonts,amscd,amssymb}
\usepackage[color,notref,notcite]{showkeys}
\usepackage{url}
\usepackage{subeqnarray}
\usepackage[color,notref,notcite]{showkeys} 
\definecolor{labelkey}{rgb}{0.6,0,1}
\usepackage{enumitem}
\usepackage{algorithm}
\usepackage{algpseudocode}
\usepackage{citesort}

\theoremstyle{plain}
\newtheorem{theorem}{Theorem}[section]

\newtheorem{lemma}[theorem]{Lemma}

\newtheorem{hypothesis}[theorem]{Hypothesis}

\theoremstyle{definition}
\newtheorem{definition}[theorem]{Definition}

\def\bhyp#1{\begin{equation}\label{#1}\begin{array}{c}}
\def\ehyp{\end{array}\end{equation}}

\newcounter{cst}

\theoremstyle{remark}

\numberwithin{equation}{section}
\numberwithin{figure}{section}

\newcommand{\RR}{{\mathbb R}}
\newcommand{\NN}{{\mathbb N}}

\def\dsp{\displaystyle}
\def\O{\Omega}
\def\bfn{\mathbf{n}}

\def\disc{{\mathcal D}}

\def\dr{\partial}
\newcommand{\cH}{{\mathcal H}}

\def\cF{{\mathcal F}}
\def\cD{{\mathbb D}}
\def\cP{{\mathbb P}}
\def\cQ{{\mathcal Q}}
\def\cL{{\mathcal L}}
\def\cE{{\mathbb E}}
\def\half{{\dsp\frac{1}{2}}}
\def\tn{{L^2(\cD \times (0,T))^d}}
\def\tnn{{L^2(\cD \times (0,T))}}
\def\sn{{L^2(\cD)}}
\def\ld{{\langle}} 
\def\rd{{\rangle}}
\def\tr{{{\rm Tr}}}
\def\cT{{\mathbb T}}

\newif\ifcorr\corrtrue
\usepackage[normalem]{ulem}
\definecolor{violet}{rgb}{0.580,0.,0.827}

\def\bpsi{{\boldsymbol \psi}}
\def\btheta{{\boldsymbol \theta}}

\newcommand{\ud}{\, \mathrm{d}} 
\def\div{\mathop{\rm div}}

\title [Generic Numerical Analysis of Stochastic Reaction Diffusion Model]{Generic Numerical Analysis of Stochastic Reaction Diffusion Model with applications in excitable media }

\author{Yahya Alnashri }
\address[Yahya Alnashri]{Department of Mathematics, Al-Qunfudah University College, Umm Al-Qura University, Saudi Arabia}
\email{yanashri@uqu.edu.sa}

\author{Hasan Alzubaidi}
\address[Hasan Alzubaidi]{Department of Mathematics, Al-Qunfudah University College, Umm Al-Qura University, Saudi Arabia}
\email{hmzubaidi@uqu.edu.sa}

\keywords{Stochastic reaction-diffusion model, Multiplicative noise, A gradient discretisation method (GDM), Gradient schemes, Convergence analysis, Non conforming numerical methods, Finite volume schemes, Hybrid mixed mimetic (HMM) method, Wave 
propagation failure, Backfiring of waves.}     

\date{\today}

\begin{document}
\newcommand{\subscript}[2]{$#1 _ #2$}

\begin{abstract}
The stochastic reaction-diffusion model driven by a multiplicative noise is examined. We construct the gradient discretisation method (GDM), an abstract framework combining several numerical method families. The paper provides the discretisation and proves the convergence of the approximate schemes using a compactness argument that works under natural assumptions on data. We also investigate, using a finite volume method, known as the hybrid mixed mimetic (HMM) approach, the effects of multiplicative noise on the dynamics of the travelling waves in the excitable media displayed by the model. Particularly, we consider how sufficiently high noise can cause waves to backfire or fail to propagate.
\end{abstract}

\maketitle
\raggedbottom 

\section{Introduction}
\label{introduction}
A very active research area, stochastic reaction-diffusion models are now considered essential modelling instruments for describing and analysing wave phenomena that occur in excitable media  \cite{N1,N2,N3}. A case in point is the electrical activity waves that travel through cardiac muscle to produce a heartbeat. These technologies are based on the premise that spiral waves and the onset of arrhythmias are caused by noise and irregularities  \cite{N4}. Noise-induced patterns in excitable media also include the generation and propagation of travelling pulses along nerve cells \cite{N5,N6}, and the waves of excitation in chemical reaction systems  \cite{N7}.
\paragraph{}
In the current work, we look on the following stochastic reaction diffusion model driven by a multiplicative noise:
\begin{subequations}\label{problem-srm}
\begin{align}
&\ud \bar v-\div(\nabla\bar v) \ud t=S(\bar v)\ud t+ g(\bar v)\ud B(t) \quad \mbox{ in }\quad \O_T:= \cD\times (0,T),\label{rm-strong1}\\
&\bar v = 0, \quad  \mbox{ on }\quad \dr\cD\times (0,T), \label{rm-strong2}\\
&\bar v(\cdot,0)=v_0 \quad\mbox{ in }\quad \cD, \label{rm-strong3}
\end{align}
\end{subequations}
where $T>0$, $\cD \subset \RR^d$, $d=\{1,2,3\}$, with an initial condition $v_0 \in L^2(\cD)$. Moreover, $B:=\{B(t):t\in [0,T]\}$ is a Brownian motion on a Hilbert space $\cH$ with a covariance operator $\cQ$. Precise assumptions on data will be detailed in the subsequent section.

A wealthy body of theoretical framework has been built for stochastic reaction diffusion equations; see for example \cite{N8,N9}.
However, the irregularity of the noise term that perturbs the model makes developing an analytical solution for such a model difficult. Consequently, in this case, accurate numerical solutions are essential. Models of this nature have been studied and analysed using a variety of numerical approaches,  including finite difference methods \cite{N10,N11,N12}, spectral methods \cite{N4,N13,N14} and, finite element methods \cite{N15,N16,N17}. Most earlier research, it appears, concentrated only on conforming numerical schemes, which are unable to maintain the model's physical properties. Therefore, the main objective of our work is to develop a generic numerical analysis of the equation \eqref{problem-srm} by means of the gradient discretisation method (GDM) \cite{S1}. The GDM approach involves a broad class of both conforming and non-conforming schemes. For instance, hybrid high-order and virtual elements methods, hybrid mimetic mixed methods, mixed finite element methods, and conforming and non-conforming finite element methods. Our findings are applicable to all these families of methods. As this approach removes the need to prove convergence for each framework, it greatly simplifies convergence analysis research.  

A convergence analysis based on the GDM framework has been described in the literature for the deterministic version of the stochastic reaction diffusion model ($g(\bar v)=0$). Some numerical examples based upon the finite volume approach that fit in the GDM approach were also presented to verify the theoretical outcomes  \cite{YH-DRDM-1,YH-DRDM-2}. To the best of our knowledge,
 the extant literature has not yet presented a generic numerical analysis of the stochastic reaction diffusion model  \eqref{problem-srm}. The GDM analysis was presented for a general stochastic Stefan problem with multiplicative noise \cite{J-SPDE-2} and a general stochastic evolution problem based on a Leray-Lions type operator \cite{J-SPDE-1}. In this work, we provide a comprehensive and unified convergence analysis of the stochastic reaction diffusion equation \eqref{problem-srm} using the GDM method. To the best of our knowledge, this analysis for such a model yields the first results applicable to various conforming and non-conforming approaches.  

By using a finite volume method, known as the hybrid mimetic mixed (HMM) \cite{Eymard-2010}, we also include some numerical tests to investigate and evaluate the behaviour of the travelling wave solutions of the investigated model under the influence of the multiplicative noise. Noise generally has several effects on excitable media; see \cite{N18} and cited works within for a thorough analysis. In this work, we investigate two noise-induced travelling wave dynamics phenomena that the investigated model displays using the HMM approach. The first is referred to as the phenomenon of wave propagation failure because of high noise, where no waves can be observed in this case \cite{N19}. Thus, when the noise rises, a travelling waveform at first starts to lose its shape. Thereafter, the wave is destroyed when the noise becomes strong enough \cite{N20}. The second wave dynamics phenomenon that the model under study displays is wave backfiring. In this regime, noise-related increases in travelling wave widths. The wave can be split into two halves that travel in opposite directions at the same speed if the transition from the excited to the quiescent states is induced in the middle of the wave for sufficiently large noise \cite{N21,N22}.  

The remainder of this paper is presented as follows. We present the weak martingale solutions and mathematical formulation of the model in Section \ref{sec-pre}. The approximate scheme is covered in Section \ref{sec-disc}. Section \ref{sec-est} establishes abound on numerical solutions and the noise term, which are essential to obtain the convergence results. In Section \ref{sec-theorem}, we study and prove the main results, Theorem \ref{theorem-1}. Section \ref{sec-num} gives some numerical tests that investigate the behaviour of the stochastic Allen-Cahn equation solutions as an example of the stochastic reaction-diffusion model \eqref{problem-srm}. Moreover, we investigate the influences of the multiplicative noise on the dynamics of such a model's travelling wave solutions. The last section contains our final thoughts and observations.

\section{Preliminaries}
\label{sec-pre}

\subsection*{{\bf Notations and assumptions:}} Let the spatial domain $\cD \subset \RR^d\;(d \geq 1)$ be an open bounded connected set and the triple $(\O,\cF,\cP)$ be a probability space, where $(\cF_t)_{t\in[0,T]}$ is a filtration satisfying the classical conditions. Given $T\in (0,\infty)$, let $(\O,\cF,(\cF_t)_{t\in[0,T]},\cP)$ be a stochastic basis. Respect to this stochastic basis, let $B=\{B(t): 0\leq t \leq T\}$ denotes an $\cH$-valued $\cQ$-Brownian motion with covariance operator $\cQ$ such that $\mbox{Tr}(\cQ)<\infty$, and which is of the form:
\begin{equation}\label{Wiener-eq}
B(t)=\dsp\sum_{k=0}^\infty q_K B_K(t)e_K,
\end{equation}
where $\{e_K\}_{K\in\NN}$ is an orthonormal basis of $\cH$ consisting of eigenvectors of $\cQ$ with finite sum of non negative eigenvalues $q_K$. Here $\{B_k\}_{k\in\NN}$ is a family of independent $\cF_t$--adapted scalar Brownian motion. The Banach space $\cL(\cH,L^2(\cD)):=\{ \cL:\cH \to L^2(\cD)\;:\; \cL \mbox{ is a bounded linear operator } \}$, and it is endowd with the operator norm denoted by $\|\cdot\|_{\cL(\cH,L^2(\cD))}$. Throughout this paper, $\langle \cdot,\cdot \rangle_{L^2}$ denotes the inner product in $L^2(\cD)$, and $\cE[\cdot]$ denotes the expectation operator. Moreover, we assume the following standard conditions:
\begin{hypothesis}\label{hyp-1}
The model's data is assumed to satisfy the following:
\begin{itemize}
\item Let $\cD \subset \RR^d\;(d \geq 1)$ be an open bounded connected set, and $T>0$.
\item Let $g:L^2(\cD) \to \cL(\cH,L^2(\cD))$ be a continuous functional satisfying the following condition:
\begin{equation}\label{eq-g-funct}
\begin{aligned}
&\exists g_1,g_2>0 \mbox{ such that, for all  } \varphi\in L^2(\cD),\\
&\|g(\varphi)\|_{\cL(\cH,L^2(\cD))}^2 \leq g_1\|\varphi\|_{L^2(\cD)}^2+g_2.
\end{aligned}
\end{equation}
\item Let the reaction function $S$ satisfy the Lipschitz continuity condition with a positive Lipschitz constant $L$. 
\item $v_0 \in L^2(\cD)$.

\end{itemize}
\end{hypothesis}

\subsection*{{\bf Weak formulation:}} Under the above assumptions on model data, a solution to the stochastic equation \eqref{problem-srm} can be understood in a variational form. Let the solution be a stochastic process taking values in $H_0^1(\cD)$. This means
\[
\bar v:[0,T] \times \cD \to H_0^1(\cD), \quad (t,\omega) \mapsto \bar v(t,\cdot,s). 
\]
Multiplying the equation \eqref{problem-srm} with a test function $\varphi \in H_0^1(\cD)$, integrating over the time-space domain and performing integration by parts over the spatial variable generates the weak solution of our model, as in the following definition.

\begin{definition}
Assume that $v_0 \in L^2(\O,\cF_0,\cP,L^2(\cD))$. A continuous  $\bar v:[0,T]\times\cD \to H_0^1(\cD)$ is a weak solution of \eqref{problem-srm} if
\begin{itemize}
\item the process $\bar v$ is a progressively measurable,
\item $\bar v(\cdot,\theta)\in C([0,T];L_{\rm w}^2(\cD))$ for $\mathbb P$-a.s. $\theta\in\O$, where $L_{\rm w}^2(\cD)$ is equipped with a weak topology,
\item $\cE\Big[\dsp\sup_{t\in[0,T]}\|\bar v(t)\|_{L^2(\cD)}^2\Big] <\infty$ and
$\cE\Big[\|\bar v\|_{L^2(0,T;H_0^1(\cD))}^2\Big] <\infty$,
\item the following inequality holds $\mathbb P$-almost surely for all $t\in[0,T]$ and for all $\varphi\in H_0^1(\cD)$,
\begin{equation}\label{rm-weak}
\begin{aligned}
\langle \bar v(t),\varphi \rangle_{L^2}
&-\langle v_0,\varphi \rangle_{L^2}
+\dsp\int_0^t \langle \nabla\bar v(s),\nabla\varphi \rangle_{L^2} \ud s\\
&=\dsp\int_0^t \langle g(\bar v)(s,\cdot)\ud B(s),\varphi \rangle_{L^2}
+\dsp\int_0^t \langle S(\bar v(s)),\varphi \rangle_{L^2} \ud s,
\end{aligned}
\end{equation}
in which the first integral in the RHS is the It{\^o} integral in $L^2(\cD)$. 
\end{itemize}
\end{definition}

\section{Approximate Scheme}\label{sec-disc}
Here, we use the gradient discretisation method to build our generic analysis. To construct this framework, we start by defining some discrete space and operators as follows:

\begin{definition}\label{def-gd-rm}
Let $\cD$ be an open subset of $\RR^d$ (with $d \geq 1$) and $T>0$. A gradient discretisations considered here is $\disc=(X_{\disc,0},\Pi_\disc, \nabla_\disc, J_\disc, (t^{(n)})_{n=0,...,N}) )$, where
\begin{enumerate}
\item the set of discrete unknowns $X_{\disc,0}$ is a finite dimensional vector space on $\RR$,
\item the linear mapping $\Pi_\disc : X_{\disc,0} \to L^2(\cD)$ is the reconstructed function, 
\item the linear mapping $\nabla_\disc : X_{\disc,0} \to L^2(\cD)^d$ is a reconstructed gradient, which must be chosen such that
\begin{equation}\label{norm-disc}
\| \varphi \|_\disc = \| \nabla_\disc \varphi \|_{L^2(\cD)^d} 
\end{equation}
is a norm on $X_{\disc,0}$, 
\item the linear continuous mapping $J_\disc: L^2(\cD) \to X_{\disc,0}$ is an interpolation operator for the
 initial condition,
\item $t^{(0)}=0<t^{(1)}<....<t^{(N)}=T$ are time steps. 
\end{enumerate}
\end{definition}
For $(u^{(n)})_{n\in\NN} \subset X_{\disc,0}$, the functions $\Pi_\disc u: (0,T]\to L^2(\cD)$, $\nabla_\disc u: (0,T]\to L^2(\cD)^d$, and $\delta_\disc u : (0,T] \to L^2(\cD)$ are piecewise constant in time and given by: For a.e $x\in\cD$, for all $n\in\{0,...,N-1 \}$ and for all $t\in (t^{(n)},t^{(n+1)}]$,
\begin{equation*}
\begin{aligned}
&\Pi_\disc u(x,0):=\Pi_\disc u^{(0)}(x), \quad \Pi_\disc u(x,t):=\Pi_\disc u^{(n+1)}(x),\\
&\nabla_\disc u(x,t):=\nabla_\disc u^{(n+1)}(x), \quad
\delta_\disc u(t)=\delta_\disc^{(n+\frac{1}{2})}u:=\frac{\Pi_\disc(u^{(n+1)}-u^{(n)})}{\delta t^{(n+\frac{1}{2})}},
\end{aligned}
\end{equation*}
with takink $\delta t^{(n+\frac{1}{2})}=t^{(n+1)}-t^{(n)}$ and $\delta t_\disc=\max_{n=0,...,N-1}\delta t^{(n+\frac{1}{2})}$.

Using the elements of $\disc$ in the problem \eqref{rm-weak} yields an approximate method called a gradient scheme. 

\begin{definition}
The related gradient scheme for the continuous problem \eqref{rm-weak} is to seek $v\in X_{\disc,0}$ satisfying the following inequality:
\begin{equation}\label{gs-rm-weak}
\begin{aligned}
&\langle d_\disc^{(n+\frac{1}{2})}v,\Pi_\disc\varphi \rangle_{L^2}
+\delta t_\disc \langle \nabla_\disc v^{(n+1)},\nabla_\disc\varphi \rangle_{L^2}\\
&\quad=\dsp\langle g(\Pi_\disc v^{(n)})\Delta^{(n+1)}B,\Pi_\disc\varphi \rangle_{L^2}
+\delta t_\disc \langle S(\Pi_\disc v^{(n)}),\Pi_\disc\varphi \rangle_{L^2}, \quad \varphi\in X_{\disc,0}.
\end{aligned}
\end{equation}
\end{definition}
The gradient discretisation elements must enjoy some properties to build convergent schemes, as stated below. Note that, in practice, different numerical methods encompassed in the gradient discretisation framework verify these properties. See \cite[Chapter 7]{S1}, for more details.

First, for $u_1\in H_0^1(\cD)$, let
\begin{equation}\label{cons-rm1}
S_\disc(u_1)= 
\min_{u_2\in X_{\disc,0}}\left(\| \Pi_\disc u_2 - u_1 \|_{L^{2}(\cD)}
+ \| \nabla_\disc u_2 - \nabla u_1 \|_{L^{2}(\cD)^{d}}\right).
\end{equation}
Next, for $\btheta \in H_{\rm div}:=\{\btheta \in L^2(\cD)^d\;:\; {\rm div}\bpsi \in L^2(\cD),\; \btheta\cdot\bfn=0 \mbox{ on } \dr\cD  \}$, let
\begin{equation}\label{long-rm}
W_\disc(\btheta)
 = \max_{u_2\in X_{\disc,0}\setminus \{0\}}\frac{\Big| \langle \nabla_\disc u_2,\btheta \rangle_{L^2(\cD)^d} + \langle \Pi_\disc u_2,\div (\btheta) \rangle_{L^2(\cD)} \Big|}{\| u_2\|_\disc }.
\end{equation} 

\begin{definition}
A sequence $(\disc_m)_{m\in \NN}$ of gradient discretisations is
\begin{itemize}
\item consistent if 
\begin{enumerate}
\item $\dsp\lim_{m\to\infty}S_{\disc_m}(u_1)=0$, for all $u_1 \in H_0^1(\cD)$, 
\item $\Pi_{\disc_m}J_{\disc_m}\varphi \to \varphi$ strongly in $L^2(\cD)$,, for all $u_1 \in L^2(\cD)$,
\item $\delta t_{\disc_m} \to 0$.
\end{enumerate}
\item limit--conformity if $\dsp\lim_{m\to\infty}W_{\disc_m}(\btheta)=0$, for all $\btheta \in H_{\rm div}$.
\item compact if for any sequence $(u_m )_{m\in\NN} \in X_{\disc_m}$, such that $(\| u_m \|_{\disc_m})_{m\in \NN}$ is bounded, the sequence $(\Pi_{\disc_m}u_m )_{m\in \NN}$ is relatively compact in $L^2(\cD)$. 
\end{itemize}
\end{definition} 

\section{Discrete Energy Estimates}\label{sec-est}
Using the same reasonings as in \cite{J-SPDE-1}, we establish here some estimates on solutions and their gradients. Throughout this section, the letter $C$ denotes different positive constants, which do not depend on the discretisation, and $C_0:=|S(0)|$.

\begin{lemma}
\label{lemma-1}
Let $v$ be a solution to the approximate problem \eqref{gs-rm-weak}. Then, for some positive constant $C$, we have
\begin{equation}\label{est-srm-1}
\begin{aligned}
\cE\Big[\max_{1\leq n\leq N}\|\Pi_\disc v^{(n)} \|_{L^2(\cD)}^2
&+\|\nabla_\disc v\|_{L^2(\cD \times (0,T))^d}^2\\
&+\dsp\sum_{n=0}^{N-1}\|\Pi_\disc v^{(n+1)}-\Pi_\disc v^{(n)} \|_{L^2(\cD)}^2
\Big]\\
&\leq C.
\end{aligned}
\end{equation}
\end{lemma}

\begin{proof}
It is known that
\begin{equation}\label{equality-1}
(a_1 -a_2)a_1=\frac{1}{2}(a_1^2-a_2^2)+\frac{1}{2}(a_1-a_2)^2,\quad \forall a_1,a_2 \in \RR.
\end{equation}
Taking $\varphi = v^{(n+1)}$ in \eqref{gs-rm-weak} as a test function and use the above relation to get
\begin{equation}\label{est-srm-11}
\begin{aligned}
&\half \|\Pi_\disc v^{(n+1)}\|_{\sn}^2 + \half \|\Pi_\disc v^{(n+1)}-\Pi_\disc v^{(n)}\|_{\sn}^2\\
&\qquad+\delta t_\disc \|\nabla_\disc v^{(n+1)}\|_{\tn}^2\\
&=\half \|\Pi_\disc v^{(n)}\|_{\sn}^2
+\ld g(\Pi_\disc v^{(n)})\Delta^{(n+1)} B, \Pi_\disc(v^{(n+1)}-v^{(n)}) \rd\\
&\qquad+\ld g(\Pi_\disc v^{(n)})\Delta^{(n+1)} B, \Pi_\disc v^{(n)} \rd
+\delta t_\disc\ld S(\Pi_\disc v^{(n+1)}),\Pi_\disc v^{(n+1)} \rd.
\end{aligned}
\end{equation}
Summing over $n=0,...,k$, with $k\in\{0,...,N-1\}$, and performing the Cauchy--Schwarz inequality yield 
\begin{equation*}
\begin{aligned}
&\half \|\Pi_\disc v^{(k+1)}\|_{\sn}^2
+\half\dsp\sum_{n=0}^k \|\Pi_\disc v^{(n+1)}-\Pi_\disc v^{(n)}\|_{\sn}^2\\
&\quad+\delta t_\disc\dsp\sum_{n=0}^k\|\nabla_\disc v^{(n+1)}\|_{\tn}^2\\
&\leq \half \|\Pi_\disc v^{(0)}\|_{\sn}^2
+\dsp\sum_{n=0}^k \|g(\Pi_\disc v^{(n)})\|_\cH \|\Delta^{(n+1)}B\|_\cH \|\Pi_\disc v^{(n+1)}-\Pi_\disc v^{(n)}\|_{\sn}\\
&\quad+\dsp\sum_{n=0}^k \ld g(\Pi_\disc v^{(n)})\Delta^{(n+1)}B,\Pi_\disc v^{(n)} \rd\\
&\quad+\delta t_\disc\dsp\sum_{n=0}^k\|S(\Pi_\disc v^{(n+1)})\|_{\tnn}\|\Pi_\disc v^{(n+1)}\|_{\tnn}.
\end{aligned}
\end{equation*}
Applying Young's inequality $\alpha\beta\leq \alpha^2+\frac{\beta^2}{4}$ to the second term on the RHS and the Lipschitz continuity assumption to the last term on RHS, we have
\begin{equation}
\begin{aligned}
&\half \|\Pi_\disc v^{(k+1)}\|_{\sn}^2
+\frac{1}{4}\dsp\sum_{n=0}^k \|\Pi_\disc v^{(n+1)}-\Pi_\disc v^{(n)}\|_{\sn}^2\\
&\quad+\delta t_\disc\dsp\sum_{n=0}^k\|\nabla_\disc v^{(n+1)}\|_{\tn}^2\\
&\leq \half \|\Pi_\disc v^{(0)}\|_{\sn}^2
+\dsp\sum_{n=0}^k \|g(\Pi_\disc v^{(n)})\|_\cH^2 \|\Delta^{(n+1)}B\|_\cH^2\\
&\quad+\dsp\sum_{n=0}^k \ld g(\Pi_\disc v^{(n)})\Delta^{(n+1)}B,\Pi_\disc v^{(n)} \rd\\
&\quad+\delta t_\disc\dsp\sum_{n=0}^k\Big( L\|\Pi_\disc v^{(n+1)}\|_{\tnn}^2+C_0\|\Pi_\disc v^{(n+1)}\|_{\tnn}\Big).
\end{aligned}
\end{equation}
After applying Young's inequality to the last term on the RHS, we have the following inequality
\begin{equation}\label{est-srm-7}
\begin{aligned}
&\half \|\Pi_\disc v^{(k+1)}\|_{\sn}^2
+\frac{1}{4}\dsp\sum_{n=0}^k \|\Pi_\disc v^{(n+1)}-\Pi_\disc v^{(n)}\|_{\sn}^2\\
&\qquad+\delta t_\disc\dsp\sum_{n=0}^k \|\nabla_\disc v^{(n+1)}\|_{\tn}^2\\
&\leq \half \|\Pi_\disc v^{(0)}\|_{\sn}^2
+\dsp\sum_{n=0}^k \|g(\Pi_\disc v^{(n)})\|_\cH^2 \|\Delta^{(n+1)}B\|_\cH^2\\
&\qquad+\dsp\sum_{n=0}^k \ld g(\Pi_\disc v^{(n)})\Delta^{(n+1)}B,\Pi_\disc v^{(n)} \rd\\
&\qquad+\delta t_\disc\dsp\sum_{n=0}^k\Big( \frac{2L+1}{2}\|\Pi_\disc v^{(n+1)}\|_{\tnn}^2+\frac{C_0^2}{2}\Big).
\end{aligned}
\end{equation}
It is useful to see that the expectation of $\Delta^{(n+1)}B$ is zero, which results in eliminating the third term on the RHS. Thus, taking the expectation implies to
\begin{equation}\label{est-srm-4}
\begin{aligned}
&\half \cE\Big[\|\Pi_\disc v^{(k+1)}\|_{\sn}^2
+\frac{1}{4}\dsp\sum_{n=0}^k \|\Pi_\disc v^{(n+1)}-\Pi_\disc v^{(n)}\|_{\sn}^2\Big]\\
&\qquad+\cE\Big[\dsp\sum_{n=0}^k \delta t_\disc \|\nabla_\disc v^{(n+1)}\|_{\tn}^2\Big]\\
&\leq \half \|\Pi_\disc v^{(0)}\|_{\sn}^2
+\dsp\sum_{n=0}^k \cE\Big[\|g(\Pi_\disc v^{(n)})\|_\cH^2 \|\Delta^{(n+1)}B\|_\cH^2\Big]\\
&\qquad+\cE\Big[\dsp\sum_{n=0}^k \delta t_\disc\Big( \frac{2L+1}{2}\|\Pi_\disc v^{(n+1)}\|_{\tnn}^2+\frac{C_0^2}{2}\Big)\Big].
\end{aligned}
\end{equation}
The assumption on $g$ implies that \cite[Equation 14]{J-SPDE-1} still holds. The second term on the RHS is thus expressed as
\begin{equation}\label{eq-Q1}
\cE[\|g(\Pi_\disc v^{(n)})\|_\cH^2 \|\Delta^{(n+1)}B\|_\cH^2]
\leq \delta t_\disc({\rm Tr}\cQ)(g_1\cE[\|\Pi_\disc v^{(n)}\|_{\sn}^2]+g_2).
\end{equation}
Together with \eqref{est-srm-4}, this implies
\begin{equation}\label{est-srm-5}
\begin{aligned}
&\cE\Big[\|\Pi_\disc v^{(k+1)}\|_{\sn}^2\Big]\\
&\leq \|\Pi_\disc v^{(0)}\|_{\sn}^2
+2(\tr\cQ)g_2T
+2(\tr\cQ)g_1\dsp\sum_{n=0}^k\delta t_\disc\cE\Big[\|\Pi_\disc v^{(n)}\|_{\sn}^2\Big]\\
&\qquad+\cE\Big[\dsp\sum_{n=0}^k\delta t_\disc(2L+1)\|\Pi_\disc v^{(n+1)}\|_{\tnn}^2+C_0^2\Big].
\end{aligned}
\end{equation}
The consistency property gives, in which $S_0$ is a positive constant depending on $v_0$
\[
\|\Pi_\disc v^{(0)}\|_{\sn} \leq S_0.
\]
Use this inequality and apply the discrete version of Gronwall’s lemma to \eqref{est-srm-5} to arrive at
\begin{equation}\label{est-srm-6-old}
\max_{1\leq n\leq N}\cE\Big[\|\Pi_\disc v^{(n)}\|_{\sn}^2 \Big] 
\leq C\cE\Big[\max_{1\leq n\leq N}\|\Pi_\disc v^{(n)}\|_{\sn}^2\Big].
\end{equation}
Relations \eqref{est-srm-4} and \eqref{est-srm-6-old} give
\[
\begin{aligned}
\cE\Big[\|\nabla_\disc v\|_{\tn}^2
&+\dsp\sum_{n=0}^{N-1}\|\Pi_\disc v^{(n+1)}-\Pi_\disc v^{(n)}\|_{\sn}^2
\Big]\\
&\leq C\cE\Big[\max_{1\leq n\leq N}\|\Pi_\disc v^{(n)}\|_{\sn}^2\Big].
\end{aligned}
\]
Take the supremum of \eqref{est-srm-7} on $k$ and the expectations to conclude
\begin{equation}\label{est-srm-8}
\begin{aligned}
&\cE\Big[\max_{1\leq n \leq N-1}\|\Pi_\disc v^{(n)}\|_{\sn}^2\Big]\\
&\leq \|\Pi_\disc v^{(0)}\|_{\sn}^2
+2\cE\Big[\dsp\sum_{n=0}^{N-1}\|g(\Pi_\disc v^{(n)})\|_\cH^2 \|\Delta^{(n+1)}B\|_\cH^2\Big]\\
&\qquad+2\cE\Big[\max_{0\leq k\leq N-1}\dsp\sum_{n=0}^k \ld g(\Pi_\disc v^{(n)})\Delta^{(n+1)}B,\Pi_\disc v^{(n)} \rd\Big]\\
&\qquad+\delta t_\disc\cE\Big[\max_{0\leq k\leq N-1}\dsp\sum_{n=0}^k\Big( (2L+1)\|\Pi_\disc v^{(n)}\|_{\tnn}^2+C_0^2\Big)\Big].
\end{aligned}
\end{equation}
Now, the third term on the RHS can be written as, thanks to the stochastic integral and the BurkH\"older–Davis–Gundy inequality: [11, Theorem 2.4]
\begin{equation}\label{est-srm-9}
\begin{aligned}
&\cE\Big[\max_{0\leq k\leq N-1}\dsp\sum_{n=0}^k \ld g(\Pi_\disc v^{(n)})\Delta^{(n+1)}B,\Pi_\disc v^{(n)} \rd\Big]\\
&\leq\frac{1}{4}\cE\Big[\max_{0\leq n\leq N}\|\Pi_\disc v^{(n)}\|_{\sn}^2\Big]
+C^2 g_1T \max_{0\leq n\leq N}\cE\Big[\|\Pi_\disc v^{(n)}\|_{\sn}^2\Big]\\
&\qquad+C^2g_2T.
\end{aligned}
\end{equation}
With \eqref{eq-Q1}, the second term on the RHS of \eqref{est-srm-8} is estimated as
\begin{equation}\label{est-srm-10}
\begin{aligned}
&\cE\Big[\dsp\sum_{n=0}^{N-1}\|g(\Pi_\disc v^{(n)})\|_\cH^2 \|\Delta^{(n+1)}B\|_\cH^2\Big]\\
&\leq ({\rm Tr}\cQ)g_1 T \max_{1\leq n \leq N}\cE\Big[\|\Pi_\disc v^{(n)}\|_{\sn}^2\Big]+(\tr\cQ)g_2T.
\end{aligned}
\end{equation}
Substituting estimates \eqref{est-srm-6-old}, \eqref{est-srm-9} and \eqref{est-srm-10} in inequality \eqref{est-srm-8} gives
\[
\cE\Big[\max_{1\leq n\leq N}\|\Pi_\disc v^{(n)}\|_{\sn}^2\Big]
\leq C.
\]
\end{proof}

\begin{lemma}\label{lemma-2}
Let $v$ be a solution to the approximate problem \eqref{gs-rm-weak}. Then, for some positive constant $C$, we have
\begin{equation}\label{eq-est-2}
\cE\Big[\max_{1\leq n\leq N}\|\Pi_\disc v^{(n)}\|_{\sn}^4 + \|\nabla v^{(n)}\|_{\tn}^4\Big]
\leq C.
\end{equation}
\end{lemma}

\begin{proof}
Introduce the following notation
\[
V^{(n)}:=\|\Pi_\disc v^{(n)}\|_{\tnn}^2 + 
\frac{\delta t_\disc}{2}\|\nabla_\disc v^{(n)}\|_{\tn}^2
+A_k,
\]
where $A_K$ is defined by
\[
A_k:=\dsp\sum_{i=0}^k\delta t_\disc \|\nabla_\disc v^{(i)}\|_{\tn}^2.
\] 
We begin with deriving bound on the quantity $\max_n\cE\Big[(V^{(n)})^2\Big]$. Using the above formula, Equality \eqref{est-srm-11} becomes
\begin{equation*}
\begin{aligned}
&V^{(n+1)}-V^{(n)}+\|\Pi_\disc v^{(n+1)}-\Pi_\disc v^{(n)}\|_{\sn}^2
+\frac{\delta t_\disc}{2}\|\nabla_\disc v^{(n+1)}-\nabla_\disc v^{(n)}\|_{\tn}^2\\
&\leq 2\ld g(\Pi_\disc v^{(n)})\Delta^{(n+1)} B, \Pi_\disc(v^{(n+1)}-v^{(n)}) \rd
+\ld g(\Pi_\disc v^{(n)})\Delta^{(n+1)} B, \Pi_\disc v^{(n)} \rd\\
&\quad+\delta t_\disc\|\Pi_\disc v^{(n+1)}\|_{\sn}^2 \ld S(\Pi_\disc v^{(n+1)}),\Pi_\disc v^{(n+1)}\rd.
\end{aligned}
\end{equation*}
Multiplying the above inequality by $V^{(n+1)}$ yields
\begin{equation}\label{est-srm-2-1}
\begin{aligned}
&V^{(n+1)}\Big(V^{(n+1)}-V^{(n)}\Big)+V^{(n+1)}\|\Pi_\disc v^{(n+1)}-\Pi_\disc v^{(n)}\|_{\sn}^2\\
&\leq V^{(n+1)}\delta t_\disc \ld S(\Pi_\disc v^{(n+1)}),\Pi_\disc v^{(n+1)} \rd + \cT_1 + \cT_2,
\end{aligned}
\end{equation}
where
\begin{equation*}
\begin{aligned}
&\cT_1:=2V^{(n+1)}\ld g(\Pi_\disc v^{(n)})\Delta^{(n+1)} B, \Pi_\disc(v^{(n+1)}-v^{(n)}) \rd, \mbox{ and }\\
&\cT_2:=V^{(n+1)}
\ld g(\Pi_\disc v^{(n)})\Delta^{(n+1)} B, \Pi_\disc v^{(n)} \rd.
\end{aligned}
\end{equation*}
For the estimation of $\cT_1$ and $\cT_2$, we apply the Cauchy–Schwarz and the Young inequalities. This gives
\begin{equation}\label{est-T1}
\begin{aligned}
\cT_1 &\leq 2V^{(n+1)}\|g(\Pi_\disc v^{(n)})\|_\cH^2 \|\Delta^{(n+1)} B\|_\cH^2 \|\Pi_\disc v^{(n+1)}\|_{\sn}^2\\
&\leq 2V^{(n+1)}\|g(\Pi_\disc v^{(n)})\|_\cH^2 \|\Delta^{(n+1)} B\|_\cH^2\\
&\quad+\half V^{(n+1)}\|\Pi_\disc v^{(n+1)}-\Pi_\disc v^{(n)}\|_{\sn}^2\\
&\leq 2\Big(V^{(n+1)}-V^{(n)}\Big)\|g(\Pi_\disc v^{(n)})\|_\cH^2 \|\Delta^{(n+1)} B\|_\cH^2\\
&\quad+2V^{(n)}\|g(\Pi_\disc v^{(n)})\|_\cH^2 \|\Delta^{(n+1)} B\|_\cH^2\\
&\quad+\half V^{(n+1)}\|\Pi_\disc v^{(n+1)}-\Pi_\disc v^{(n)}\|_{\sn}^2\\
&\leq \frac{1}{8}(V^{(n+1)}-V^{(n)}\Big)^2 + 8\|g(\Pi_\disc v^{(n)})\|_\cH^4 \|\Delta^{(n+1)} B\|_\cH^4\\
&\quad+2V^{(n)}\|g(\Pi_\disc v^{(n)})\|_\cH^2 \|\Delta^{(n+1)} B\|_\cH^2\\
&\quad+\half V^{(n+1)}\|\Pi_\disc v^{(n+1)}-\Pi_\disc v^{(n)}\|_{\sn}^2,
\end{aligned}
\end{equation}
and
\begin{equation}\label{est-T2}
\begin{aligned}
\cT_2 &=\ld g(\Pi_\disc v^{(n)})\Delta^{(n+1)} B, \Pi_\disc v^{(n)} \rd\\
&=\Big(V^{(n+1)}-V^{(n)} \Big)\ld g(\Pi_\disc v^{(n)})\Delta^{(n+1)} B, \Pi_\disc v^{(n)} \rd\\
&\quad+V^{(n)}\ld g(\Pi_\disc v^{(n)})\Delta^{(n+1)} B, \Pi_\disc v^{(n)} \rd\\
&\leq \frac{1}{16}\Big(V^{(n+1)}-V^{(n)} \Big)^2
+4\|g(\Pi_\disc v^{(n)})\|_\cH^2 \|\Delta^{(n+1)} B\|_\cH^2 \|\Pi_\disc v^{(n)}\|_{\sn}^2\\
&\quad+V^{(n)}\ld g(\Pi_\disc v^{(n)})\Delta^{(n+1)} B, \Pi_\disc v^{(n)} \rd.
\end{aligned}
\end{equation}
We know that
\begin{equation}\label{relation-1}
\begin{aligned}
\|\Pi_\disc v^{(n)}\|_{\sn}^2 \leq V^{(n)}, \mbox{ and }
\|g(\Pi_\disc)v^{(n)}\|_\cH^2 \leq g_1 V^{(n)}+g_2.
\end{aligned}
\end{equation}
Using the moments' bounds \eqref{est-T1} and \eqref{est-T2}, we conclude from \eqref{est-srm-2-1} that, thanks to \eqref{relation-1}
\begin{equation}\label{est-srm-2-2}
\begin{aligned}
&\half \Big(V^{(n+1)}\Big)^2 - \half \Big(V^{(n)}\Big)^2
+\frac{5}{16}\Big(V^{(n+1)}-V^{(n)}\Big)^2\\
&\quad+\half \|\Pi_\disc v^{(n+1)}-\Pi_\disc v^{(n)}\|_{\sn}^2\\
&\leq 8\Big(g_1 V^{(n)}+g_2\Big)^2 \|\Delta^{(n+1)} B\|_\cH^4 
+10 V^{(n)}\Big(g_1 V^{(n)}+g_2\Big) \|\Delta^{(n+1)} B\|_\cH^2\\
&\quad+V^{(n)}\ld g(\Pi_\disc v^{(n)})\Delta^{(n+1)} B, \Pi_\disc v^{(n)} \rd\\
&\quad+V^{(n+1)}\delta t_\disc \ld S(\Pi_\disc v^{(n+1)}),\Pi_\disc v^{(n+1)} \rd.
\end{aligned}
\end{equation}
We use the Cauchy–Shwarz and the Young inequalities, the Lipschitz continuity condition, and \eqref{relation-1} to deal with the last term on the RHS. It follows that
\begin{equation}\label{est-srm-2-3}
\begin{aligned}
&\half \Big(V^{(n+1)}\Big)^2 - \half \Big(V^{(n)}\Big)^2
+\frac{5}{16}\Big(V^{(n+1)}-V^{(n)}\Big)^2\\
&\quad+\half \|\Pi_\disc v^{(n+1)}-\Pi_\disc v^{(n)}\|_{\sn}^2\\
&\leq 8\Big(g_1 V^{(n)}+g_2\Big)^2 \|\Delta^{(n+1)} B\|_\cH^4 
+10 V^{(n)}\Big(g_1 V^{(n)}+g_2\Big) \|\Delta^{(n+1)} B\|_\cH^2\\
&\quad+V^{(n)}\ld g(\Pi_\disc v^{(n)})\Delta^{(n+1)} B, \Pi_\disc v^{(n)} \rd\\
&\quad+V^{(n+1)}\delta t_\disc \Big( \frac{2L+1}{2}V^{(n+1)}+\frac{C_0^2}{2}\Big).
\end{aligned}
\end{equation}
The expectation of the third term on the RHS equals zero. With relation \eqref{relation-1}, the first two terms on the RHS can be estimated as
\begin{equation}\label{est-srm-2-4}
\begin{aligned}
&8\cE\Big[\Big(g_1V^{(n)}+g_2\Big)^2 \|\Delta^{(n+1)} B\|_\cH^4\Big] 
+10\cE\Big[V^{(n)}\Big(g_1V^{(n)}+g_2\Big) \|\Delta^{(n+1)} B\|_\cH^2\Big]\\
&\leq 8\delta t_\disc^2(\tr(\cQ))^2\cE\Big[(g_1V^{(n)}+g_2)^2\Big]
+10\delta t_\disc\tr(\cQ)\cE\Big[V^{(n)}(g_1V^{(n)}+g_2)].
\end{aligned}
\end{equation}
Using the above expression, the inequality \eqref{est-srm-1}, and the discrete version of Gronwall's lemma, summing \eqref{est-srm-2-3} over $n=0,...,k$, and taking the expectations imply to
\begin{equation}\label{est-srm-6}
\max_{1\leq n\leq N}\cE\Big[\Big(V^{(n)}\Big) \Big]
\leq C.
\end{equation}
We perform the formula \eqref{equality-1} to \eqref{est-srm-2-1} and sum over $n=0,...,k$ to conclude that, thanks to the estimates \eqref{est-T1} and \eqref{est-T2}
\[
\begin{aligned}
\half\Big(v^{(k+1)}\Big)^2
&\leq \half \Big(v^{(0)}\Big)^2 + 8\dsp\sum_{n=0}^k\|g(\Pi_\disc)v^{(n)}\|_\cH^4 \|\Pi_\disc v^{(n)}\|_{\sn}^4\\
&\quad+2\dsp\sum_{n=0}^k\|g(\Pi_\disc)v^{(n)}\|_\cH^2\|\Delta^{(n+1)} B\|_\cH^2\Big(V^{(n)}\Big)\\
&\quad+4\dsp\sum_{n=0}^k\|g(\Pi_\disc)v^{(n)}\|_\cH^2\|\Delta^{(n+1)} B\|_\cH^2
\|\Pi_\disc v^{(n)}\|_{\sn}^2\\
&\quad+\dsp\sum_{n=0}^k\ld g(\Pi_\disc v^{(n)})\Delta^{(n+1)} B, \Pi_\disc v^{(n)} \rd\Big(V^{(n)}\Big)\\
&\quad+\dsp\sum_{n=0}^k V^{(n+1)}\delta t_\disc \ld S(\Pi_\disc v^{(n+1)}),\Pi_\disc v^{(n+1)} \rd.
\end{aligned}
\] 
Hence, we take a maximum over $k$ and then take expectation to attain
\begin{equation}\label{est-srm-2-5}
\begin{aligned}
&\cE\Big[\max_{1\leq n\leq N-1}\Big(V^{(k+1)}\Big)^2\Big]\\
&\leq\Big(V^{(0)}\Big)^2 + 16\cE\Big[\dsp\sum_{n=0}^{N-1}\|g(\Pi_\disc)v^{(n)}\|_\cH^4 \|\Pi_\disc v^{(n)}\|_{\sn}^4\Big]\\
&\quad+4\cE\Big[\dsp\sum_{n=0}^{N-1}\|g(\Pi_\disc)v^{(n)}\|_\cH^2\|\Delta^{(n+1)} B\|_\cH^2\Big(V^{(n)}\Big)\Big]\\
&\quad+8\cE\Big[\dsp\sum_{n=0}^{N-1}\|g(\Pi_\disc)v^{(n)}\|_\cH^2\|\Delta^{(n+1)} B\|_\cH^2
\|\Pi_\disc v^{(n)}\|_{\sn}^2\Big]\\
&\quad+2\cE\Big[\max_{0\leq n\leq N-1}\dsp\sum_{n=0}^k\ld g(\Pi_\disc v^{(n)})\Delta^{(n+1)} B, \Pi_\disc v^{(n)} \rd\Big(V^{(n)}\Big)\Big]\\
&\quad+\cE\Big[\dsp\sum_{n=0}^{N-1} V^{(n+1)}\delta t_\disc \ld S(\Pi_\disc v^{(n+1)}),\Pi_\disc v^{(n+1)} \rd\Big],
\end{aligned}
\end{equation} 
Applying the BurkH\"older–Davis–Gundy inequality (with constant $C$), one has
\begin{equation}\label{est-srm-2-6}
\begin{aligned}
&2\cE\Big[\dsp\sum_{n=0}^{N-1}\ld g(\Pi_\disc v^{(n)})\Delta^{(n+1)} B, \Pi_\disc v^{(n)} \rd\Big(V^{(n)}\Big)\Big]\\
&\leq 2C\cE\Big[\Big(\dsp\sum_{n=0}^{N-1}\delta t_\disc \|g(\Pi_\disc)v^{(n)}\|_\cH^2 \|\Pi_\disc v^{(n)}\|_{\sn}^2 \Big(V^{(n)}\Big)^2\Big)^\half\Big]\\
&\leq 2C\cE\Big[\max_{0\leq n\leq N-1}V^{(n)}\Big(\dsp\sum_{n=0}^{N-1}\delta t_\disc (g_1 V^{(n)}+g_2)(V^{(n)})\Big)^\half\Big]\\
&\leq \half\cE\Big[\max_{0\leq n\leq N-1}(V^{(n)})^2\Big]
+4C^2\dsp\sum_{n=0}^{N-1}\delta t_\disc\cE\Big[g_1(V^{(n)})^2+g_2\Big]\\
&\leq \half\cE\Big[\max_{0\leq n\leq N-1}(V^{(n)})^2\Big]
+4C^2T\max_{0\leq n\leq N-1}\Big(\cE\Big[(V^{(n)})^2\Big]
+g_2\Big).
\end{aligned}
\end{equation} 
The last term on the RHS of \eqref{est-srm-2-5} can be introduced as
\begin{equation}\label{est-srm-2-7}
\begin{aligned}
&\cE\Big[\dsp\sum_{n=0}^{N-1} V^{(n)}\delta t_\disc \ld S(\Pi_\disc v^{(n+1)}),\Pi_\disc v^{(n+1)} \rd \Big]\\
&\leq \cE\Big[\dsp\sum_{n=0}^k \delta t_\disc V^{(n)}\|\Pi_\disc v^{(n)}\|_{\sn}\|S(\Pi_\disc v^{(n)})\|_{\sn} \Big]\\
&\leq \cE\Big[\max_{0\leq n\leq N-1}(V^{(n)})\dsp\sum_{n=0}^{N-1} \delta t_\disc\|\Pi_\disc v^{(n)}\|_{\sn}\|S(\Pi_\disc v^{(n)})\|_{\sn}
\Big]\\
&\leq\cE\Big[\max_{0\leq n\leq N-1}(V^{(n)})^2+\Big]
+C\max_{0\leq n\leq N-1}\cE\Big[(V^{(n)})^2\Big]+C
\Big].
\end{aligned}
\end{equation} 
The estimates \eqref{est-srm-2-3}, \eqref{est-srm-6}, \eqref{est-srm-2-5}, \eqref{est-srm-2-6}, and \eqref{est-srm-2-7} lead to
\[
\cE\Big[\max_{0\leq n\leq N-1}(V^{(n)})^2\Big]
\leq C,
\]
which completes the proof of the assertion.
\end{proof}

\begin{lemma}\label{lemma-3}
If $v$ is a solution the approximate problem \eqref{gs-rm-weak}, then, for some positive constant $C$, we have, for every $r \in \{1,...,N-1\}$
\begin{equation}
\cE\Big[\delta t_\disc\dsp\sum_{n=0}^{N-r}\|\Pi_\disc v^{(n+r)}-\Pi_\disc v^{(n)}\|_{\sn}^2\Big] \leq t^{(r)}C.
\end{equation}
\end{lemma}

\begin{proof}
For an arbirtry $i\in\{0,...,r\}$, we sum \eqref{gs-rm-weak} over $i$ to obtain
\begin{equation}\label{est-srm-3-1}
\begin{aligned}
\delta t_\disc\|\Pi_\disc v^{(n+r)}-\Pi_\disc v^{(n)}\|_{\sn}^2
=\cT_1 + \cT_2+\cT_3,
\end{aligned}
\end{equation}
where
\begin{equation}\label{est-srm-3-2}
\begin{aligned}
&\cT_1:=-\delta t_\disc^2\dsp\sum_{n=1}^{N-r}\dsp\sum_{i=0}^{r-1}\ld \nabla_\disc v^{(n+i+1)}, \nabla_\disc v^{(n+r)}-\nabla_\disc v^{(n)} \rd,\\
&\cT_2:=\delta t_\disc\dsp\sum_{n=1}^{N-r}\dsp\sum_{i=0}^{r-1}\ld g(\Pi_\disc v^{(n+i)})\Delta^{(n+i+1)} B, \Pi_\disc v^{(n+r)}-\Pi_\disc v^{(n)} \rd, \mbox{ and }\\
&\cT_3:=\delta t_\disc^2\dsp\sum_{n=1}^{N-r}\dsp\sum_{i=0}^{r-1}\ld S(\Pi_\disc v^{(n+i+1)}), \Pi_\disc v^{(n+r)}-\Pi_\disc v^{(n)} \rd.
\end{aligned}
\end{equation}
We make use of Cauchy--Schwarz inequalities to get a bound on the expectation of $\cT_1$ as follows 
\begin{equation}\label{est-srm-3-3}
\begin{aligned}
\cE[\cT_1]&\leq\cE
\Big[
\delta t_\disc\dsp\sum_{n=1}^{N-r} \|\nabla_\disc(v^{(n+r)}-v^{(n)})\|_{L^2(\cD)^d}\dsp\int_{t^{(n)}}^{t^{(n+r)}}\|\nabla_\disc v\|_{L^2(\cD)^d}
\Big]\\
&\leq \Big(t^{(r)}\Big)^{\frac{1}{2}}\cE\Big[
\Big(\delta t_\disc\dsp\sum_{n=1}^{N-r} \|\nabla_\disc(v^{(n+r)}-v^{(n)})\|_{L^2(\cD)^d}^2\Big)^{\frac{1}{2}}\\
&\quad\times
\Big(\dsp\sum_{n=1}^{N-r}\delta t_\disc\dsp\int_{t^{(n)}}^{t^{(n+r)}}\|\nabla_\disc v\|_{L^2(\cD)^d}^2\Big)^{\frac{1}{2}}
\Big]\\
&\leq\Big(t^{(r)}\Big)^{\frac{1}{2}}
\cE\Big[
\Big(2\dsp\int_{t^{(r)}}^{t^{(N)}}\|\nabla v\|_{L^2(\cD)^d}^2
+2\dsp\int_{t^{(r)}}^{t^{(N)}}\|\nabla v\|_{L^2(\cD)^d}^2\Big)^{\frac{1}{2}}\\
&\quad+\Big(r\delta t_\disc\dsp\int_{t^{(1)}}^{t^{(N)}}\|\nabla v\|_{\sn}^2\Big)^{\frac{1}{2}}
\Big]\\
&\leq 2\Big(t^{(r)}\Big)^{\frac{1}{2}}
\cE\Big[\|\nabla_\disc v\|_{\tn}^2\Big].
\end{aligned}
\end{equation}
We similarly use Cauchy--Schwarz inequalities to estimate the term $\cT_3$, thanks to Lipschitz continuity assumption 
\begin{equation}\label{est-srm-3-6}
\begin{aligned}
\cE[\cT_3]
&=\cE\Big[\delta t_\disc^2\dsp\sum_{n=1}^{N-r}\dsp\sum_{i=0}^{r-1}\ld S(\Pi_\disc v^{(n+i+1)}), \Pi_\disc v^{(n+r)}-\Pi_\disc v^{(n)} \rd\Big]\\
&\leq \cE\Big[\delta t_\disc^2\dsp\sum_{n=1}^{N-r}\|\Pi_\disc(v^{(n+r)}-v^{(n)})\|_{\sn}\dsp\sum_{i=0}^{r-1}\|S(\Pi_\disc v^{(n+i+1)})\|_{\sn}\Big]\\
&\leq \cE\Big[\delta t_\disc^2\dsp\sum_{n=1}^{N-r}\|\Pi_\disc(v^{(n+r)}-v^{(n)})\|_{\sn}\dsp\sum_{i=0}^{r-1}L\|\Pi_\disc v^{(n+i+1)}\|_{\sn}+C_0\Big]\\
&\leq \cE\Big[\delta t_\disc\dsp\sum_{n=1}^{N-r}\|\Pi_\disc(v^{(n+r)}-v^{(n)})\|_{\sn}dsp\sum_{i=0}^{r-1}\|S(\Pi_\disc v^{(n+i+1)})\|_{\sn}\Big]\\
&\leq \cE\Big[\delta t_\disc^2\dsp\sum_{n=1}^{N-r}\|\Pi_\disc(v^{(n+r)}-v^{(n)})\|_{\sn}\dsp\sum_{i=0}^{r-1}L\|\Pi_\disc v^{(n+i+1)}\|_{\sn}+C_0\Big]\\
&\leq\cE\Big[\delta t_\disc\dsp\sum_{n=1}^{N-r}\|\Pi_\disc(v^{(n+r)}-v^{(n)})\|_{\sn}\dsp\int_{t^{(n)}}^{t^{(n+r)}}L\|\Pi_\disc v\|_{\sn}+C_0
\Big]\\
&\leq \Big(t^{(r)}\Big)^{\frac{1}{2}}\cE\Big[\delta t_\disc\Big(\dsp\sum_{n=1}^{N-r}\|\Pi_\disc(v^{(n+r)}-v^{(n)})\|_{\sn}^2\Big)^{\frac{1}{2}}\\
&\quad\times
\Big(\dsp\sum_{n=1}^{N-r}\dsp\int_{t^{(n)}}^{t^{(n+r)}}L\|\Pi_\disc v\|_{\sn}^2\Big)^{\frac{1}{2}}+C_0
\Big]\\
&\leq \Big(t^{(r)}\Big)^{\frac{1}{2}}\Big(\delta t_\disc r\Big)^{\frac{1}{2}}\cE\Big[\|\Pi_\disc u\|_{\sn}^2\Big].
\end{aligned}
\end{equation}
For the estimation of $\cT_2$, we apply Young's inequality to attain
\begin{equation}\label{est-srm-3-4}
\begin{aligned}
&\cE[\cT_2]\\
&= \delta t_\disc\dsp\sum_{n=1}^{N-r}\cE\Big[
\ld\dsp\int_0^T \chi_{[t^{(n)},t^{(n+r)}]}(t)g(\Pi_\disc v(t))dB(t),
\Pi_\disc v^{(n+r)}-\Pi_\disc v^{(n)}\rd
\Big]\\
&\leq \frac{1}{4}\delta t_\disc\dsp\sum_{n=1}^{N-r}\cE\Big[\|\Pi_\disc v^{(n+r)}-\Pi_\disc v^{(n)}\|_{\sn}^2\Big]\\
&\quad+\delta t_\disc \dsp\sum_{n=1}^{N-r}\cE\Big[\|\dsp\int_0^T\chi_{[t^{(n)},t^{(n+r)}]}(t)g(\Pi_\disc v(t))dB(t)\|_{\sn}^2\Big].
\end{aligned}
\end{equation}
Because of the It\^o isometry, \eqref{eq-g-funct} and Lemma \ref{lemma-1}, the last term on the RHS is estimated as
\begin{equation}\label{est-srm-3-5}
\begin{aligned}
&\cE\Big[\|\dsp\int_0^T\chi_{[t^{(n)},t^{(n+r)}]}(t)g(\Pi_\disc v(t))dB(t)\|_{\sn}^2\Big]\\
&\leq(\tr\cQ)\cE\Big[
\dsp\int_0^T\chi_{[t^{(n)},t^{(n+r)}]}(t)\|g(\Pi_\disc v(t))\|_\cH^2\|
\Big]\\
&\leq(\tr\cQ)\cE\Big[
\dsp\int_0^T\chi_{[t^{(n)},t^{(n+r)}]}(t)(g_1\|\Pi_\disc v(t)\|_{\sn}^2+g_2)
\Big]\\
&\leq(\tr\cQ)t^{(r)}\cE\Big[g_1\max_{1\leq n\leq N}\|v^{(n)}\|_{\sn}^2+g_2\Big]\\
&\leq C\|v^{(0)}\|_{\sn}t^{(r)}.
\end{aligned}
\end{equation}
Together with \eqref{est-srm-3-3}, \eqref{est-srm-3-4}, \eqref{est-srm-3-5} and \eqref{est-srm-3-6} complete the proof.
\end{proof}

\begin{definition}\label{def-dual}
The dual norm $| \cdot |_{\star,\disc}$ is defined by: for all $v\in \Pi_\disc(X_{\disc,0})$,
\begin{equation}\label{eq-d-norm}
|v|_{\disc,\star}:=\dsp\sup\Big\{\dsp\int_\cD u(x)\Pi_\disc \psi(x) \ud x\;:\; \psi\in X_{\disc,0} \mbox{ and }\|\nabla_\disc \psi\|_{L^2(\cD)^d} = 1 \Big\}.
\end{equation}
\end{definition}

\begin{lemma}\label{lemma-4}
Let $v$ be a solution of \eqref{gs-rm-weak}. Then, for all $r=1,...,N-1$, there exists a positive constant $C$ satisfying
\begin{equation}\label{eq-d-est-1}
\cE\Big[|\Pi_\disc v^{(n+r)}-\Pi_\disc v^{(n)}|_{\disc,\star}^4\Big]
\leq C\left(t^{(r)}\right)^2.
\end{equation}
Also, for any $t_1,t_2\in [0,T]$, we have
\begin{equation}\label{eq-d-est-2}
\cE\Big[|\Pi_\disc v(t_1)-\Pi_\disc v(t_2)|_{\disc,\star}^4\Big]
\leq C(|t_1-t_2|+\delta t_\disc)^2.
\end{equation}
\end{lemma}

\begin{proof}
Taking an arbitrary $w$, such that $w\in X_{\disc,0}, \|\nabla_\disc w\|_{L^2(\cD)^d} = 1$, the relation \eqref{est-srm-3-1} can be rewritten as
\begin{equation}\label{est-d-norm-1}
\begin{aligned}
\cE\Big[|\Pi_\disc v^{(n+r)}-\Pi_\disc v^{(n)}|_{\disc,\star}^4\Big]
&=\cE\Big[|\dsp\sup_{w}\ld \Pi_\disc v^{(n+r)}-\Pi_\disc v^{(n)},\Pi_\disc w \rd|^4\Big]\\
&\leq \cT_1 + \cT_2 + \cT_3,
\end{aligned}
\end{equation}
where
\begin{equation*}
\begin{aligned}
&\cT_1:=2^3(\delta t_\disc)^4\cE\Big[\Big(\dsp\sup_{w}\dsp\sum_{i=0}^{r-1}\ld\nabla_\disc v^{(n+i+1)},\nabla_\disc w\rd\Big)^4\Big], \\
&\cT_2:=2^3\cE\Big[\Big(\dsp\sup_{w}\dsp\sum_{i=0}^{r-1}\ld g(\Pi_\disc v^{(n+i)})\Delta^{(n+i+1)}B,\Pi_\disc w\rd\Big)^4\Big],\mbox{ and }\\
&\cT_3:=2^3(\delta t_\disc)^4\cE\Big[\Big(\dsp\sup_{w}\dsp\sum_{i=0}^{r-1}\ld S(\Pi_\disc v^{(n+i+1)}),\Pi_\disc w\rd\Big)^4\Big]
\end{aligned}
\end{equation*}
To get a bound on $\cT_1$, we apply Lemma \ref{lemma-2} and the H\"older inequality
\begin{equation}\label{est-d-norm-2}
\begin{aligned}
\cT_1 &\leq C (\delta t_\disc)^4 \cE\Big[ \Big( \dsp\sum_{i=0}^{r-1}\|\nabla_\disc v^{(n+i+1)}\|_{L^2(\cD)^d}\sup_{w}\|\nabla_\disc w\|_{L^2(\cD)^d} \Big)^4 \Big]\\
&\leq C \cE\Big[ \Big( \dsp\int_{t^{(n)}}^{t^{(n+r)}}\|\nabla_\disc v\|_{L^2(\cD)^d}\Big)^4 \Big]\\
&\leq C(t^{(r)}) \cE\Big[ \Big( \dsp\int_{t^{(n)}}^{t^{(n+r)}}\|\nabla_\disc v\|_{L^2(\cD)^d}^2\Big)^2 \Big]\\
&\leq C(t^{(r)})\cE\Big[ \Big\|\nabla_\disc v\|_{\tn}^2 \Big] \leq c(t^{(r)}).
\end{aligned}
\end{equation}
For the estimation of $\cT_2$, we employ the BurkH\"older--Davis--Gundy inequality, \eqref{eq-g-funct} and \eqref{lemma-1}
\begin{equation}\label{est-d-norm-3}
\begin{aligned}
\cT_2 &\leq \cE\Big[\|\dsp\int_0^T \chi_{[t^{(n)},t^{(n+r)}]}(t)g(\Pi_\disc v(t))d B(t)\|_{\sn}^4 \Big]\\
&\leq \cE\Big[\Big(\dsp\int_0^T \chi_{[t^{(n)},t^{(n+r)}]}(t) (g_1\|\Pi_\disc v(t)\|_{\sn}^2+g_2)\Big)^2 \Big]\\
&\leq C(t^{(r)})^2\cE\Big[g_1\max_{1\leq n\leq N}\|\Pi_\disc u^{(n)}\|_{\sn}^4+g_2\Big]
\leq Ct^{(r)}.
\end{aligned}
\end{equation}
We make use the Lipschitz continuity condition and Lemma \eqref{lemma-1} to estimate $\cT_3$
\begin{equation}\label{est-d-norm-4}
\begin{aligned}
\cT_3 &\leq C(t^{(r)})^2\cE\Big[LC_0\sup_{w}\|\Pi_\disc w\|_{\sn}^2\Big]\\
&\quad+C\delta t_\disc \cE\Big[L\sup_{w}\|\Pi_\disc w\|_{\sn}^2\Big( \dsp\sum_{i=0}^{r-1}\|\Pi_\disc v^{(n+i+1)}\|_{\sn}^2 \Big) \Big]\\
&\leq C(t^{(r)})^2 + C(t^{(r)})\cE\Big[ \Big( \dsp\int_{t^{(n)}}^{t^{(n+r)}} \|\Pi_\disc v(t)\|_{\sn}^2 \Big)^4 \Big]\\
&\leq C(t^{(r)})^2 + C(t^{(r)}) \cE\Big[\|\Pi_\disc v\|_{\tnn}^2\Big]\\
&\leq C(t^{(r)})^2 + C(t^{(r)}) \Big(\cE\Big[\|\Pi_\disc v\|_{\tnn}^4\Big]\Big)^2 \leq C(t^{(r)}).
\end{aligned}
\end{equation}
Altogether \eqref{est-d-norm-1}, \eqref{est-d-norm-2}, \eqref{est-d-norm-3} and \eqref{est-d-norm-4} yield the estimate \eqref{eq-d-est-1}. The estimate \eqref{eq-d-est-2} can be obtained by using the fact that when $t_1,t_2\in [0,T]$, $t_1 < t_2$, $n\leq r$, and $t_1\in(t^{(2)},t^{(3)}]$, then $t^{(2-n)}\leq |t_2-t_1|+\delta t_\disc$.
\end{proof}

Let $v$ be a solution to the approximate problem \eqref{gs-rm-weak}. For every $0\leq t\leq T$, we have $n \in \{0,....,N-1\}$ such that $t^{(n)} \leq t \leq  t^{(n+1)}$. We introduce
\begin{equation}
G_\disc(t):=G_\disc^{(n)}:=\dsp\sum_{i=0}^ng(\Pi_\disc v^{(n)})\Delta^{(i+1)}B.
\end{equation} 
As proved in \cite{J-SPDE-1}, we can establish an estimate on $G_\disc$.

\begin{lemma}\label{lemma-5}
For any $\beta \in (0,\frac{1}{2})$, the following inequalities hold 
\begin{equation}\label{eq-est-MD}
\cE\Big[\|G_\disc\|_{H^\infty(0,T,\sn)}^4\Big] \leq C 
\mbox{ and } \cE\Big[\|G_\disc\|_{H^{\beta,4}(0,T,\sn)}^2\Big] \leq C,
\end{equation}
where $C$ is a positive constant depending on $\beta$.
\end{lemma}

\begin{proof}
The BurkH\"older–Davis--Gundy gives
\[
\begin{aligned}
\cE\Big[ \|G_\disc\|_{H^\infty(0,T,\sn)}^4\Big]&=\cE\Big[\max_{1\leq n\leq N}\Big\|\dsp\sum_{i=0}^n g(\Pi_\disc v^{(i)})\Delta^{(i+1)}B\Big\|_\sn^4 \Big]\\
&\leq \cE\Big[\Big(\dsp\int_o^T \|g(\Pi_\disc v)\|_{L(K,\sn)}^2\Big)^2 \Big].
\end{aligned}
\]
\end{proof}
Hence, the first estimate is a consequence of \eqref{eq-g-funct} and Lemma \ref{eq-est-2}. Similar to \eqref{est-d-norm-4}, we have
\begin{equation}\label{eq-dif-MD}
\cE\Big[\|G_\disc^{(n+r)}-G_\disc^{(n)}\|_{\sn}^4\Big]
\leq C(t^{(r)})^2.
\end{equation}
Reasoning as in \cite[Lemma 7.2]{J-SPDE-1}, the second estimate follows from the above inequality and the first estimate.

\section{Main Results}\label{sec-theorem}
In \cite[Section 4]{J-SPDE-1}, it is proved that the sequence $(\Pi_{\disc_m}v_m, \nabla_{\disc_m}v_m,G_{\disc_m},B)_{m\in\NN}$ is tight. Now, using our estimates established in the previous section, we can show that the sequence $(\Pi_{\disc_m}v_m, \nabla_{\disc_m}v_m,G_{\disc_m},B)_{m\in\NN}$ converges almost sure up to a change of probability space. The proof of the following lemma is identical to \cite[Lemma 4.2]{J-SPDE-1} ((slightly adjusted to the fact that $\bar v$ is a solution to the model considered here).

\begin{lemma}\label{lemma-6}
We have a probability space $( \overline\O, \overline\cF,\overline{\mathbb F},\overline{\mathbb P})$, $(\widehat v_m, \overline{Z}_m,\overline{B}_m)_{m\in\NN}$, and $(\bar v,\overline Z,\overline B)$, in which the following items are fulfilled:
\begin{itemize}
\item $\widehat v_m \in X_{\disc_m,0},\; \forall m\in\NN$,
\item the laws of $(\Pi_{\disc_m}\widehat v_m,\nabla_{\disc_m}\widehat v_m),\overline{Z}_m,\overline{B}_m)$ and $(\Pi_{\disc_m}v_m,\nabla_{\disc_m}v_m),Z_m,{B}_m)$ are the same, $\forall m\in\NN$
\item $(\bar v,\overline Z,\overline B)$ has values in $L^2(0,T;H_0^1(\cD))\times L^2(0,T;L_{\rm w}^2(\cD))\times C([0,T];L^2(\cD))$,
\item up to a subsequence, the following convergences hold $\mathbb P$-almost surely,
\begin{equation}\label{eq-conv-1}
\Pi_{\disc_m}\widehat v_m \to \bar v \mbox{ strongly in } L^2(0,T;L^2(\cD)),
\end{equation}
\begin{equation}\label{eq-conv-2}
\nabla_{\disc_m}\widehat v_m \to \nabla\bar v \mbox{ weakly in } L^2(0,T;L^2(\cD)^d),
\end{equation}
\begin{equation}\label{eq-conv-3}
\overline{Z}_m \to \overline Z \mbox{ strongly in } L^2(0,T;L_{\rm w}^2(\cD)),
\end{equation}
\begin{equation}\label{eq-conv-4}
\overline{B}_m \to \overline B \mbox{ strongly in } C([0,T];L^2(\cD)),
\end{equation}
\item $\widehat{v}_m$ satisfies the approximate problem \eqref{gs-rm-weak}, where $B=\overline{B}_m$, and the following convergences up to a subsequence holds for almost all $0<r,s<T$, 
\begin{equation}\label{eq-conv-5}
\dsp\lim_{m\to\infty}\|(\Pi_{\disc_m}\widehat v_m(r)-\Pi_{\disc_m}\widehat v_m(s))-(\bar v(r)-\bar v(s))\|_{L^2}=0,
\end{equation}
\begin{equation}\label{eq-conv-6}
\dsp\lim_{m\to\infty}\|(\overline{Z}_m(s) - \overline Z_m(r))-(\overline{Z}(r)- \overline{Z}(s))\|_{L^2}=0.
\end{equation}
\end{itemize}
\end{lemma}

\begin{theorem}\label{theorem-1}
Assume that Hypotheses \ref{hyp-1} hold. Let $(\disc_m)_{m\in\NN}$ be a sequence of gradient discretisations satisfying the consistency, limit--conformity and compactness properties. Then, for any $m\geq 1$, the approximate scheme \eqref{gs-rm-weak} (in which $\disc:=\disc_m$) has a random process $v_m$ solution. 

Also, we have a weak martingale solution $(\widetilde\O,\widetilde\cF,(\widetilde\cF_t)_{t\in[0,T]},\widetilde{\mathbb P},\widetilde B,\widetilde v)$ satisfying \eqref{rm-weak}, and a sequence $\{\widetilde v_m\}$ of random processes defined on $\widetilde\O$, which is of the same law as $v_m$, such that up to a subsequence, the following convergences are satisfied $\widetilde{\mathbb P}$-almost surely:
\begin{equation*}
\begin{aligned}
&\Pi_{\disc_m}\widetilde v_m \to \widetilde v, \quad \mbox{strongly in } L^2(\cD\times(0,T)),\\
&\nabla_{\disc_m}\widetilde v_m \to \nabla\widetilde v, \quad \mbox{weakly in } L^2(\cD\times(0,T))^d.
\end{aligned}
\end{equation*} 
\end{theorem}

\begin{proof}
Following the same reasoning as in \cite[Corollary 5.2]{YH-DRDM-1} and using the estimate \eqref{est-srm-1}, we can prove the existence of at least one solution $v$ to the scheme \eqref{gs-rm-weak}. 

Now, consider the smooth countable set $\{\psi_i \neq 0 \;:\; i\in\NN\}$ to introduce
\begin{equation}\label{psi-def}
w_i:=\dsp\frac{\psi_i}{(\|\psi_i\|_{\sn}+\|\nabla \psi_i\|_{L^2(\cD)^d})}.
\end{equation}
Let $I_\disc:H_0^1(\cD)\cap L^2(\cD) \to X_{\disc,0}$ be an interpolator defined by
\begin{equation}\label{interpol-def}
I_\disc w:=\arg\min_{\varphi\in X_{\disc,0}}\Big(\|\Pi_\disc \varphi - w\|_{\sn} + \|\nabla_\disc \varphi - \nabla w\|_{L^2(\cD)^d}\Big).
\end{equation}
Applying the consistency property with $\varphi=0$ together with $\|w_i\|_{\sn}+\|\nabla w_i\|_{L^2(\cD)^d}=1$, one has, thanks to \eqref{psi-def} and \eqref{interpol-def} 
\begin{equation}\label{est-interpol-1}
\begin{aligned}
&\|\Pi_\disc I_\disc w_i\|_{\sn} + \|\nabla_\disc I_\disc w_i\|_{L^2(\cD)^d}\\
&\leq \|\Pi_\disc I_\disc w_i - w_i\|_{\sn} + \|\nabla_\disc I_\disc w_i - \nabla w_i\|_{L^2(\cD)^d}\\
&\quad+ \|w_i\|_{\sn} + \|\nabla w_i\|_{L^2(\cD)^d}\\
&\leq S_\disc(w_i)+\|w_i\|_{\sn} + \|\nabla w_i\|_{L^2(\cD)^d}\\
&\leq 2.
\end{aligned}
\end{equation}
Thanks to Lemmas \ref{lemma-1} and \ref{lemma-4}, and the convergence stated in \eqref{eq-conv-5}, \cite[Lemmas 4.3 and 5.1]{J-SPDE-1} asserts that the process $\overline{Z}$ is a square integrable continuous martingale, and there exists a quadratic variation of $\overline Z$ given by for all $\alpha,\beta\in \sn$,
\[
Z(t)(\alpha,\beta)=\dsp\int_0^t \ld (g(\bar v)\cQ^{\frac{1}{2}})^\star(\alpha),(g(\bar v)\cQ^{\frac{1}{2}})^\star(\beta) \rd_K \ud s,\; \forall t\geq 0. 
\]
\cite[Theorem 8.2]{B-1} shows the existence of a probability space $(\widehat\cD,\widehat{\cF},\widehat{\mathbb P})$ and a $\cQ-$ Brownian motion $\widehat B$, in which $\overline Z$ and $\bar v$ are defined on this space and, for any $t\geq 0$,
\begin{equation}\label{eq-M}
\overline Z(t,\cdot)=\dsp\int_0^t g(\bar v)(s,\cdot)d\widehat B(s).
\end{equation}
For any $t\in [0,T]$ and for each $m\in\NN$, we have  a natural number $0\leq  k \leq N_{m-1}$, such that $t\in (t^{(k)},t^{(k+1)}]$. Recalling that $\bar v_m$ solves the gradient scheme with $\overline{B}_m$ replaced by $\overline B$, we have
\[
\begin{aligned}
&\ld d_{\disc_m}^{(n+\frac{1}{2})}\widehat v_m,\Pi_{\disc_m}\varphi \rd
+ \delta t_{\disc_m}\ld \nabla_{\disc_m}\widehat v_m^{(n+1)},\nabla_{\disc_m}\varphi \rd\\
&=\ld g(\Pi_{\disc_m}\widehat v_m^{(n+1)})\Delta^{(n+1)}\overline{B}_m,\Pi_{\disc_m}\varphi \rd
\quad+\delta t_{\disc_m}\ld S(\Pi_{\disc_m}\widehat v_m^{(n+1)}),\Pi_{\disc_m}\varphi \rd.
\end{aligned}
\]
Summing this relation from $n=0$ to $n=k$, and choosing $\varphi:=I_{\disc_m}\psi$, where $\psi \in H_0^1(\cD)$, we obtain,
\begin{equation}\label{eq-5-1}
\begin{aligned}
&\ld \Pi_{\disc_m}\widehat v_m(t),\Pi_{\disc_m}I_{\disc_m}\psi \rd
-\ld \Pi_{\disc_m}v^{(0)},\Pi_{\disc_m}I_{\disc_m}\psi \rd\\
&\quad+\dsp\sum_{n=0}^k \delta t_{\disc_m}\ld \nabla_{\disc_m}\widehat v_m^{(n+1)},\nabla_{\disc_m}I_{\disc_m}\psi \rd\\
&=\ld \widehat{M}_m,\Pi_{\disc_m}I_{\disc_m}\psi \rd
+\dsp\sum_{n=0}^k \delta t_{\disc_m}\ld S(\Pi_{\disc_m}\widehat v_m^{(n+1)}),\Pi_{\disc_m}I_{\disc_m}\psi \rd.
\end{aligned}
\end{equation}
Since $\Pi_{\disc_m}I_{\disc_m}\psi \to \psi$ in $L^2(\cD)$ and $(\disc_m)_{m\in\NN}$ is consistent, we have for almost every $t$, thanks to \eqref{eq-conv-5} and \eqref{eq-conv-6}
\begin{equation}\label{eq-5-2}
\begin{aligned}
&\ld \Pi_{\disc_m}\widehat v_m(t),\Pi_{\disc_m}I_{\disc_m}\psi \rd
\to \ld \widehat v(t),\psi \rd
\mbox{ in } L^2(\widehat\cD)\\
&\ld \Pi_{\disc_m}\widehat v^{(0)},\Pi_{\disc_m}I_{\disc_m}\psi \rd
\to \ld \widehat v_0,\psi \rd
\mbox{ in } L^2(\widehat\cD)\\
&\ld \widehat Z_m(t),\Pi_{\disc_m}I_{\disc_m}\psi \rd
\to \ld \widehat Z(t),\psi \rd
\mbox{ in } L^2(\widehat\cD). 
\end{aligned}
\end{equation}
Now, we notice that
\begin{equation}\label{eq-5-3}
\begin{aligned}
&\dsp\sum_{n=0}^k \delta t_{\disc_m}\ld \nabla_{\disc_m}\widehat v_m^{(n+1)},\nabla_{\disc_m}I_{\disc_m}\psi \rd\\
&=\dsp\int_0^t \ld \nabla_{\disc_m}\widehat v_m(s),\nabla_{\disc_m}I_{\disc_m}\psi \rd \ud s
+\dsp\int_t^{\lceil t/\delta t_{\disc_m} \rceil\delta t_{\disc_m}}
\ld \nabla_{\disc_m}\widehat v_m(s),\nabla_{\disc_m}I_{\disc_m}\psi \rd \ud s,
\end{aligned}
\end{equation}
and
\begin{equation}\label{eq-5-3-a}
\begin{aligned}
&\dsp\sum_{n=0}^k \delta t_{\disc_m}\ld S(\Pi_{\disc_m}\widehat v_m^{(n+1)}),\Pi_{\disc_m}I_{\disc_m}\psi \rd\\
&=\dsp\int_0^t \ld S(\Pi_{\disc_m}\widehat v_m(s)),\Pi_{\disc_m}I_{\disc_m}\psi \rd \ud s
+\dsp\int_t^{\lceil t/\delta t_{\disc_m} \rceil\delta t_{\disc_m}}
\ld S(\Pi_{\disc_m}\widehat v_m(s)),\Pi_{\disc_m}I_{\disc_m}\psi \rd \ud s,
\end{aligned}
\end{equation}
We make use the convergence stated in \eqref{eq-conv-2}, $\nabla_{\disc_m}I_{\disc_m}\psi \to \nabla\psi$ in $L^2(\cD)^d$, and $\Pi_{\disc_m}I_{\disc_m}\psi \to \psi$ in $L^2(\cD)^d$ to obtain, for any $t\in[0,T]$
\begin{equation}\label{eq-5-4}
\begin{aligned}
\dsp\int_0^t \ld \nabla_{\disc_m}\widehat v_m(s),\nabla_{\disc_m}I_{\disc_m}\psi \rd \ud s \to 
\dsp\int_0^t \ld \nabla\widehat v_m(s),\nabla\psi \rd \ud s,
\end{aligned}
\end{equation}
and
\begin{equation}\label{eq-5-4-a}
\begin{aligned}
\dsp\int_0^t \ld S(\Pi_{\disc_m}\widehat v_m(s)),\Pi_{\disc_m}I_{\disc_m}\psi \rd \ud s \to 
\dsp\int_0^t \ld S(\widehat v_m(s)),\psi \rd \ud s.
\end{aligned}
\end{equation}
These prove the convergence of the first term of the RHS of \eqref{eq-5-3} (resp. \eqref{eq-5-3-a}). Finally, as shown below, we observe that the expectation of the absolute value of the last term in the RHS of \eqref{eq-5-3} (resp. \eqref{eq-5-3-a}) tends to zero as $m\to\infty$, which concludes the proof.
\begin{equation}\label{eq-5-5}
\begin{aligned}
&\cE\Big[ \Big| \dsp\int_t^{\lceil t/\delta t_{\disc_m} \rceil\delta t_{\disc_m}}
\ld \nabla_{\disc_m}\widehat v_m(s),\nabla_{\disc_m}I_{\disc_m}\psi \rd \ud s \Big| \Big]\\
&\leq \cE\Big[  \dsp\int_t^{\lceil t/\delta t_{\disc_m} \rceil\delta t_{\disc_m}}
\| \nabla_{\disc_m}\widehat v_m(s)\|_{\tn} \| \nabla_{\disc_m}I_{\disc_m}\psi \|_{\tn} \ud s \Big]\\
&\leq C\delta t_{\disc_m}^{1/2}\cE\Big[ \Big( \dsp\int_0^T
\| \nabla_{\disc_m}\widehat v_m(s)\|_{\tn}^2 \ud s \Big)^{1/2} \Big]\\
&\leq C\delta t_{\disc_m}^{1/2},
\end{aligned}
\end{equation}
and
\begin{equation}\label{eq-5-5-a}
\begin{aligned}
&\cE\Big[ \Big| \dsp\int_t^{\lceil t/\delta t_{\disc_m} \rceil\delta t_{\disc_m}}
\ld S(\Pi_{\disc_m}\widehat v_m(s)),\Pi_{\disc_m}I_{\disc_m}\psi \rd \ud s \Big| \Big]\\
&\leq \cE\Big[  \dsp\int_t^{\lceil t/\delta t_{\disc_m} \rceil\delta t_{\disc_m}}
\| S(\Pi_{\disc_m}\widehat v_m(s))\|_{\sn} \| \Pi_{\disc_m}I_{\disc_m}\psi \|_{\sn} \ud s \Big]\\
&\leq \cE\Big[  \dsp\int_t^{\lceil t/\delta t_{\disc_m} \rceil\delta t_{\disc_m}}
L\| \Pi_{\disc_m}\widehat v_m(s)\|_{\sn} \| \Pi_{\disc_m}I_{\disc_m}\psi \|_{\sn} \ud s \Big]\\
&\quad+C_0\delta t_{\disc_m}\| \Pi_{\disc_m}I_{\disc_m}\psi \|_{\sn}\\
&\leq C\delta t_{\disc_m}^{1/2}\cE\Big[ \Big( \dsp\int_0^T
L\| \Pi_{\disc_m}\widehat v_m(s)\|_{\sn}^2 \ud s \Big)^{1/2} \Big]
+C_0\delta t_{\disc_m}\\
&\leq C(\delta t_{\disc_m}+\delta t_{\disc_m}^{1/2}).
\end{aligned}
\end{equation}
\end{proof}

\section{Numerical results}\label{sec-num}
To evaluate the validity of the gradient scheme for the stochastic reaction-diffusion model, we investigate a particular variant of a numerical method that falls in the gradient discretisation method; it is called the hybrid mimetic mixed method (HMM). We advise the reader to \cite[Section 4.2]{YH-DRDM-1} for the full presentation of this method for deterministic reaction-diffusion equations. As a test, we use the HMM method to solve the equation \eqref{problem-srm} with non-homogeneous Dirichlet boundary conditions over the domain $\cD=[-5,5]^2$, and with the reaction term $S(\bar v)=\bar v-{\bar v}^3$. The initial MATLAB code for the HMM scheme is provided as a companion of the work in \cite{HMM-code}. The model with a reaction function of this nature is referred to as the stochastic Allen-Cahn equation, and it has been the subject of extensive research in the literature. For instance,  \cite{N23,N24,N25,N26}, and the references provided there illustrate this. We initially select the noise term $g(\bar v)=\rho \bar v$, where  $\rho$ denotes the noise intensity. The exact travelling wave solution in such a case is given as \cite{N27}
\begin{equation}\label{eq14}
              \bar v(t,x,y)=(- \frac{c}{6}+2\sqrt{\frac{\rho^{2}+2}{8}-\frac{c^2}{48}} \tanh(\sqrt{\frac{\rho^{2}+2}{8}-\frac{c^2}{48}}(x+y-ct)))e^{\rho B(t)-\frac{\rho^2}{2}t},
\end{equation}
where $c$ is the wave speed and is chosen as $c=0.3$ in our numerical experiments. The analytical solution presented in \eqref{eq14} is used to extract the initial and boundary conditions.

Simulations are conducted on a sequence of triangular and Kershowa meshes, with a noise intensity of $\rho=0.2$, up to $T=1$. Samples of these meshes are presented in Figure \ref{fig1}. 
The averages are calculated over $100$ realizations of the Brownian motion, with time steps of $\delta t= 0.0001, 0.00025, 0.001$,  and $0.00125$, respectively,
where $\delta t=0.0001$ for the coarsest mesh and  $\delta t=0.00125$ for the finest one. The relative $L_2$-errors on the solution $\bar v$ and on the gradient $\nabla\bar v$ are respectively examined by the following norms:
\[
\frac{\cE\Big[\| \bar v(\cdot) - \Pi_\disc v\|_{L^{2}(\cD)}\Big]}{\cE\Big[\| \bar v(\cdot)\|_{L^{2}(\cD)}\Big]}
\mbox{ and }
\frac{\cE\Big[\| \nabla\bar v(\cdot) - \nabla_{\disc}v \|_{L^{2}(\cD)^2}\Big]}{\cE\Big[\| \nabla\bar v(\cdot) \|_{L^{2}(\cD)^2}\Big]},
\]
where $\cE[\cdot]$ denotes the expectation operator. 
The obtained errors and the corresponding rate of convergence with regard to the mesh size $h$ are given in Table \ref{tab1} for the triangular mesh and, in Table \ref{tab2} for the Kershowa mesh. 
We have observed that the HMM's convergence rate is of order $1$, which is consistent with the expectation of convergence rate for the first order methods. 

In addition, we investigate the validity of the HMM method by investigating the influences of multiplicative noise on the dynamics of the simulated waves that the stochastic Allen-Cahn model demonstrates. Two examples of such noise-induced waves dynamics are displayed in Figure \ref{fig2} and Figure \ref{fig3}.

Figure \ref{fig2} presents the HMM travelling wave solutions of the underlying model at $T=1$ for various noise intensity levels, designated as $\rho=0$, $\rho=1$, $\rho=5$, and $\rho=10$. The simulation was conducted on a triangular mesh with the size of $h=0.15625$ and a time step of $0.001$. In the absence of noise ($\rho=0$), the model exhibits a travelling wave with a maximal value of approximately $1$. Nevertheless, the wave begins to lose its wave-form shape when the noise is incorporated, and annihilates completely at high noise levels, as illustrated in Figure \ref{fig2} (c) for $\rho=5$ and Figure \ref{fig2} (d) for $\rho=10$. Consequently, the wave propagation is largely unaffected by weak noise, whereas it is unable to propagate in the strong noise regime. This phenomenon is known as \textit{ wave propagation failure} in which no propagation wave is observed.  

Additionally, we further evaluate the HMM method in the presence of another multiplicative noise source, denoted by 
$$g(\bar v)=\rho \bar v(\bar v-1),$$ 
where $\rho$ is the noise intensity. The intriguing aspect of this form of noise is that the fluctuations do not influence the model in either the rest state ($\bar v=0$) or the excited state ($\bar v=1$), as the noise vanishes at both states. The initial and boundary conditions are determined using the exact travelling wave solution \eqref{eq14} for the deterministic case ($\rho=0$). The simulation results are depicted in Figure \ref{fig3} for various levels of noise intensity, $\rho=0$, $\rho=1$, $\rho=2$, and $\rho=3$, while maintaining the other parameters, as in Figure \ref{fig2}. The results indicate that a single travelling wave is produced in the deterministic state ($\rho=0$). Nevertheless, \textit{the phenomenon of waves backfiring} can occur when an appropriate amount of noise is added. The transition from the excited state to the rest can be induced in the midst of the wave for sufficiently large noise ($\rho=3$), resulting in the wave breaking into two halves that are moving in opposite directions, as illustrated in Figure \ref{fig3} (d).        
 
\begin{figure}[ht]
	\begin{center}
	\begin{tabular}{cc}
	\includegraphics[width=0.40\linewidth]{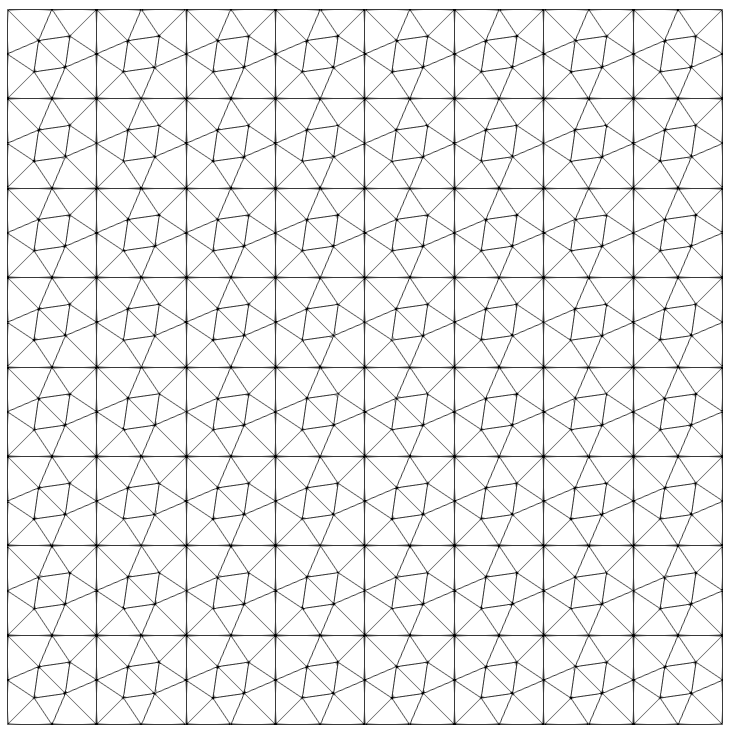}&
	\includegraphics[width=0.40\linewidth]{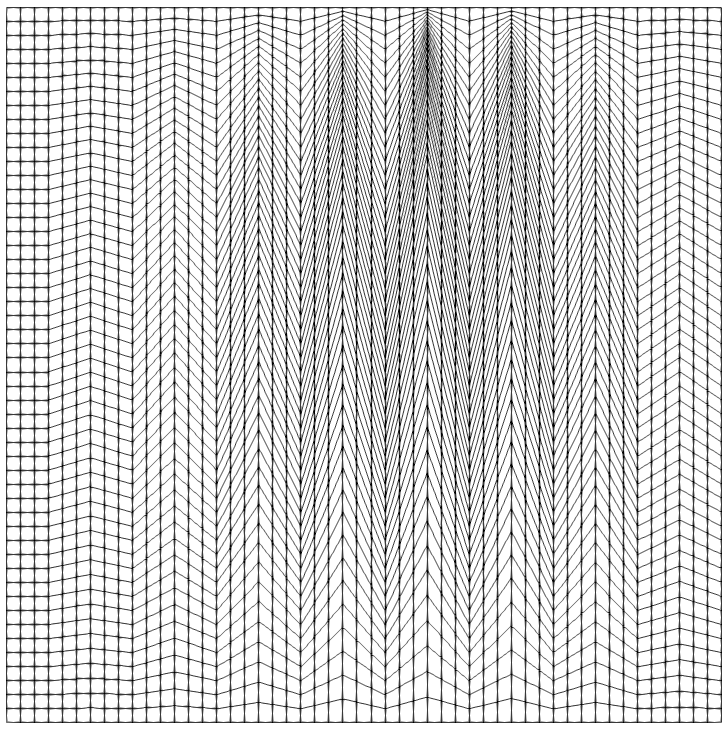}\\
	\texttt{(a) triangular mesh}&\texttt{(b) Kershowa mesh}\\
	\end{tabular}
	\end{center}
	\caption{Samples of the various 2D meshes}
	\label{fig1}
\end{figure}
\begin{figure}[ht]
	\begin{center}
	\begin{tabular}{cc}
	\includegraphics[width=0.40\linewidth]{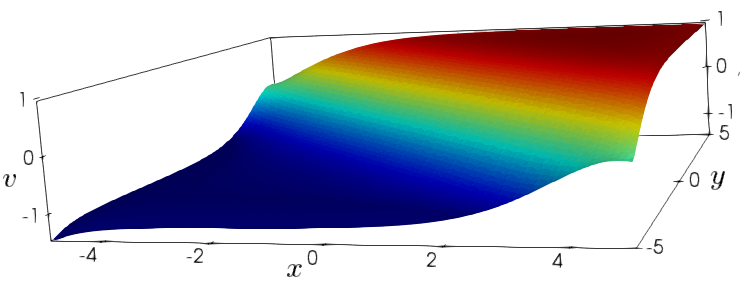}&
	\includegraphics[width=0.40\linewidth]{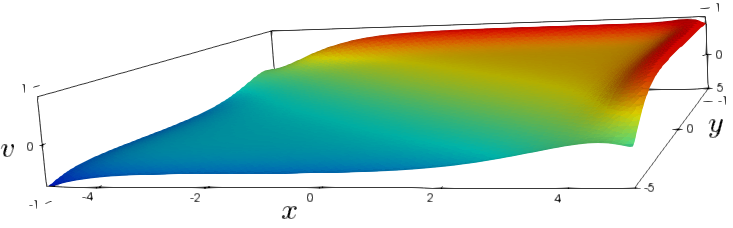}\\
	\texttt{(a) $\sigma=0$}&\texttt{(b) $\sigma=1$}\\
	\includegraphics[width=0.40\linewidth]{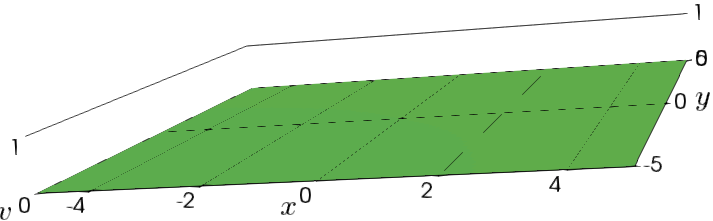}&
	\includegraphics[width=0.40\linewidth]{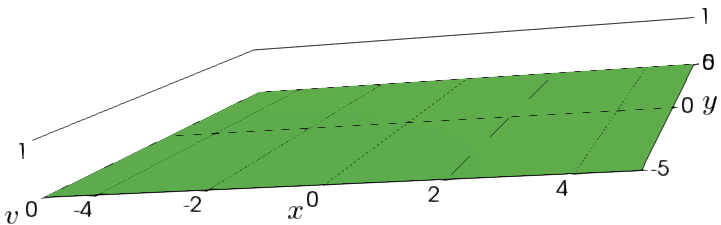}\\
	\texttt{(c) $\sigma=5$ }&\texttt{(d) $\sigma=10$ }\\
	\end{tabular}
	\end{center}
	\caption{Travelling wave solutions on the range $[-5,5]^2$, taking at levels of  noise $\rho=0$, $\rho=1$, $\rho=5$ and $\rho=10$, with parameters chosen as $\delta t=0.001$, $T=1$, and $c=0.3$.}
	\label{fig2}
\end{figure}
\begin{figure}[ht]
	\begin{center}
	\begin{tabular}{cc}
	\includegraphics[width=0.40\linewidth]{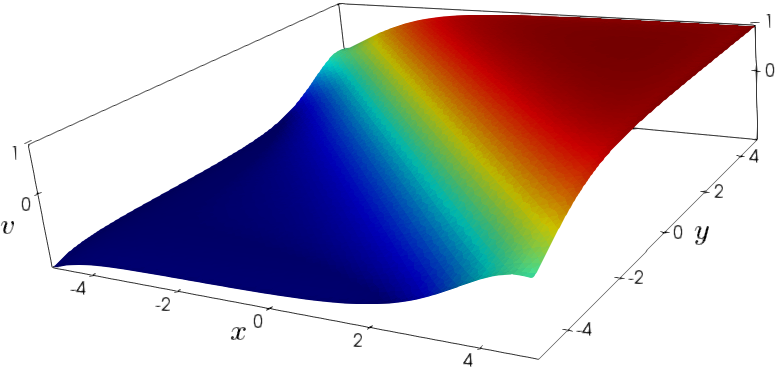}&
	\includegraphics[width=0.40\linewidth]{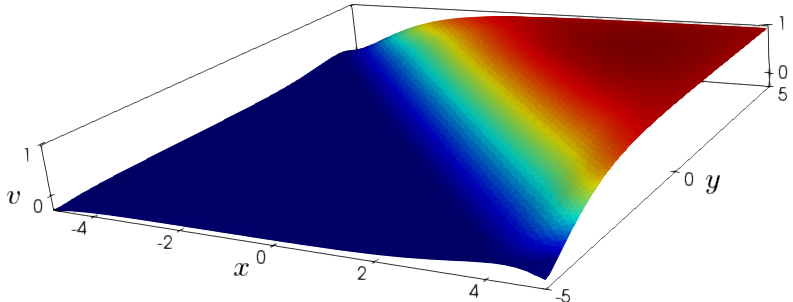}\\
	\texttt{(a) $\sigma=0$}&\texttt{(b) $\sigma=1$}\\
	\includegraphics[width=0.40\linewidth]{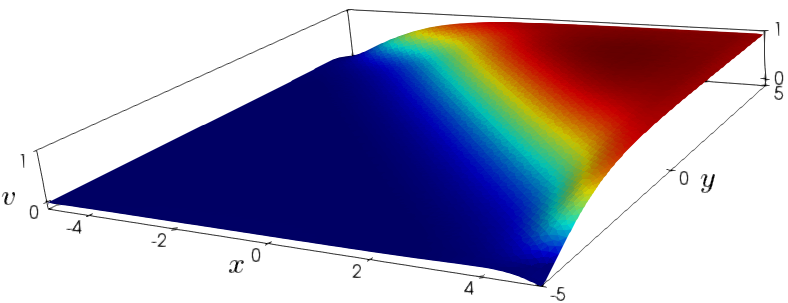}&
	\includegraphics[width=0.40\linewidth]{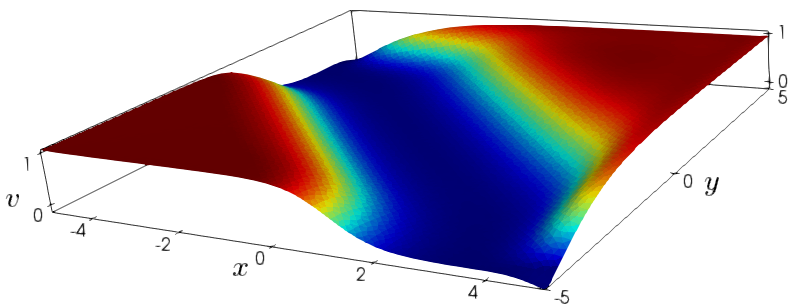}\\
	\texttt{(c) $\sigma=2$ }&\texttt{(d) $\sigma=3$ }\\
	\end{tabular}
	\end{center}
	\caption{Travelling wave solutions on the range $[-5,5]^2$, taking at levels of  noise $\rho=0$, $\rho=1$, $\rho=2$ and $\rho=3$, with parameters chosen as $\delta t=0.001$, $T=1$, and $c=0.3$.}
\label{fig3}
\end{figure}

\begin{table}[]
\begin{tabular}{c c c c c}
\hline
$h$ &
$\frac{\cE[\| \bar v(\cdot) - \Pi_\disc v\|_{L^{2}(\cD)}]}{\cE[\| \bar v(\cdot)\|_{L^{2}(\cD)}]}$&
rate&
$\frac{\cE[\| \nabla\bar v(\cdot) - \nabla_{\disc}v \|_{L^{2}(\cD)^2}]}{\cE[\| \nabla\bar v(\cdot) \|_{L^{2}(\cD)^2}]}$&
rate 
\\ \hline
0.6250000&
0.4708741&
--&
1.1986664&
--
\\ 
0.3125000&
0.2286052&
1.1986664&
0.6831530&
1.0424831
\\ 
0.1562500&
0.1299039&
0.8154136&
0.3359330&
1.0240353
\\ 
0.0781250&
0.0644115&
1.0120541&
0.1783213&
0.9136945
\\ \hline
\end{tabular}
\caption{The relative errors on $\bar v$ and $\nabla\bar v$ and the convergence rates w.r.t.  $h$ the size of triangular mesh with noise intensity $\rho=0.2$.}
\label{tab1}   
\end{table}
\begin{table}[]
\begin{tabular}{c c c c c}
\hline
$h$ &
$\frac{\cE[\| \bar v(\cdot) - \Pi_\disc v\|_{L^{2}(\cD)}]}{\cE[\| \bar v(\cdot)\|_{L^{2}(\cD)}]}$&
rate&
$\frac{\cE[\| \nabla\bar v(\cdot) - \nabla_{\disc}v \|_{L^{2}(\cD)^2}]}{\cE[\| \nabla\bar v(\cdot) \|_{L^{2}(\cD)^2}]}$&
rate 
\\ \hline
1.1155656&
0.8482250&
--&
1.2973389&
--
\\ 
0.8385224&
0.4732923&
2.0437227&
1.0346060&
0.7926934;
\\ 
0.6717051&
0.3026483&
2.0157696&
0.8655410&
0.8043447
\\ 
0.5602472&
0.2117210&
1.9692431&
0.7127184&
1.0706944
\\ \hline
\end{tabular}
\caption{The relative errors on $\bar v$ and $\nabla\bar v$ and the convergence rates w.r.t. $h$ the size of Kershowa mesh with noise intensity $\rho=0.2$.}
\label{tab2}   
\end{table}

\section{Conclusion} 
The GDM was used to develop a generic numerical analysis for a stochastic reaction-diffusion model driven by a multiplicative noise. Based on natural assumptions, we proved the existence of a weak martingale solution and established convergence results that are applicable to various numerical schemes. As a result of our analysis, we were able to develop an HMM method for the underlying model and investigate the effects of multiplicative noise on the behaviour of the solutions we obtained through numerical testing. We observed that the HMM's convergence rate is of order $1$, which matches the expectation of convergence rate for the first order methods. We also investigated the influences of this multiplicative noise on the dynamics of the travelling waves in the model under investigation, particularly the occurrence of wave propagation failure and wave backfiring as a result of sufficiently large noise.


\bibliographystyle{siam}
\bibliography{SRDM-ref}

\end{document}